\documentclass[11pt]{amsart}
\usepackage{wasysym,amssymb,amsmath,eufrak,indentfirst,graphicx,cite,amsthm,color}
\usepackage[bookmarksnumbered, colorlinks, linktocpage, plainpages]{hyperref}
\usepackage{enumitem}
\usepackage{extarrows}
\usepackage{epsfig}
\usepackage{float}
\usepackage{indentfirst}
\usepackage{color,xcolor}
\newtheorem{theorem}{Theorem}[section]
\newtheorem{corollary}[theorem]{Corollary}
\newtheorem{lemma}[theorem]{Lemma}
\newtheorem{proposition}[theorem]{Proposition}
\theoremstyle{definition}
\newtheorem{definition}[theorem]{Definition}

\theoremstyle{remark}
\newtheorem{remark}[theorem]{Remark}

\numberwithin{equation}{section}
\begin{document}

\title[Growth and Distortion  Theorems]{The Growth and Distortion  Theorems for Slice Monogenic Functions}

\thanks{This work was supported by the NNSF  of China (11371337).}

\author[G. B. Ren]{Guangbin Ren}
\address{Guangbin Ren, School of Mathematical Sciences, University of Science and
Technology of China, Hefei 230026, China}
\email{rengb$\symbol{64}$ustc.edu.cn}

\author[X. P. Wang]{Xieping Wang}
\address{Xieping Wang, School of Mathematical Sciences, University of Science and
Technology of China, Hefei 230026,
China}
\email{pwx$\symbol{64}$mail.ustc.edu.cn}


\begin{abstract}
The  sharp growth and distortion theorems  are established for   slice monogenic  extensions of univalent functions on the unit disc $\mathbb D\subset \mathbb C$ in the setting of Clifford algebras,  based on a new convex combination  identity.
The analogous results are also valid  in the   quaternionic setting for slice regular functions and we can even prove the Koebe type one-quarter theorem in this case.
Our growth and distortion theorems  for  slice regular  (slice monogenic) extensions to higher dimensions  of univalent holomorphic
functions  hold without extra geometric assumptions,
   in contrast to the setting of several complex variables in which the
growth and distortion theorems fail in general and hold only for some subclasses
with  the starlike or
convex assumption.

\end{abstract}
\maketitle

\section{Introduction}
In geometric function theory of holomorphic functions of one complex variable, the following well-known growth and distortion theorems  (cf. \cite{Duren, GG})   remark the beginning of the systematic  study of univalent functions.
\begin{theorem}[Growth and Distortion Theorems]\label{Th:dg-theorem}
Let $F$ be a univalent function on the open unit disk
$\mathbb D=\big\{z\in \mathbb C :|z| < 1\big\}$ such that $F(0)=0$ and $F'(0)=1$. Then for each $z\in\mathbb D$, the following inequalities hold:
\begin{eqnarray}\label{eq:1}
\frac{|z|}{(1+|z|)^2}\leq |F(z)|\leq \frac{|z|}{(1-|z|)^2};
\end{eqnarray}
\begin{eqnarray}\label{eq:2}
\frac{1-|z|}{(1+|z|)^3}\leq |F'(z)|\leq \frac{1+|z|}{(1-|z|)^3};
\end{eqnarray}

\begin{eqnarray}\label{eq:3}
\frac{1-|z|}{1+|z|}\leq \bigg|\frac{zF'(z)}{F(z)}\bigg|\leq \frac{1+|z|}{1-|z|}.
\end{eqnarray}
Moreover, equality holds for one of these six inequalities at some point $z_0\in \mathbb D\setminus\{0\}$ if and only if $F$ is a rotation of the Koebe function, i.e. $$F(z)=\frac{z}{(1-e^{i\theta} z)^2}\ ,\qquad \forall\,\, z\in \mathbb D,$$ for some $ \theta \in \mathbb R$.
\end{theorem}

The extension of geometric function theory to higher dimensions was suggested by H. Cartan \cite{Ca} in 1933.
But, the first meaningful   result  was only made in 1991 by Barnard, Fitzgerald and Gong \cite{BFG}.
Since then, the geometric function theory in several complex variables  has been  extensively studied, see \cite{Gong,GG} and the references therein. In particular, the growth theorem holds  for starlike  mappings  on starlike  circular domains  \cite{LR2},
and for convex   mappings  on convex  circular domains \cite{LR1}.

However, as far as we know, nearly nothing has been done about  the corresponding  theory  for other  classes of functions,  such as the classical regular (monogenic) functions in the sense of Cauchy-Fueter and the recently introduced slice regular (slice monogenic) functions,  perhaps due to the failure  of closeness  under multiplication and composition brought from non-commutativity of the underlying algebras on which these functions are defined.

In this paper, we  shall focus on   slice regular and slice monogenic functions and  aim  to generalize Theorem \ref{Th:dg-theorem} to the noncommutative setting for slice regular and slice monogenic extensions of univalent functions on the unit disc $\mathbb D\subset \mathbb C$. The theory of slice regular functions of one quaternionic variable was initiated recently by Gentili and Struppa \cite{GS1, GS2}, and was also extended by the same authors to octonions in \cite{GS50} for octonionic slice regular functions.
The related theory of slice monogenic functions on domains in the paravector space $\mathbb R^{n+1}$ with values in the Clifford algebra $\mathbb R_n$ was introduced in \cite{Co, Co1}. To have a more complete insight, we refer the reader to the monographs \cite{GSS, Co2} and the references therein. These function theories were also unified and generalized in \cite{Ghiloni1} by means of the concept of slice functions on the so-called quadratic cone of a real alternative *-algebra, based on a  slight modification of a well-known construction due to Fueter. The theory of slice regular functions on  real alternative *-algebras  is by now well-developed through a series of papers mainly due to Ghiloni and Perotti after their seminal work \cite{Ghiloni1}.
It is also well worth mentioning that this recently introduced theory of slice regular (slice monogenic) functions is significantly different from the more classical theory of regular (monogenic) functions in the sense of Cauchy-Fueter (cf. \cite{BDS, CSSS, GHS}), and has  elegant  applications to the functional calculus for noncommutative operators \cite{Co2}, to Schur analysis \cite{ACS},  and to the construction and classification of orthogonal complex structures on dense open subsets of $\mathbb R^4\simeq \mathbb H$ \cite{GSS2014}.

We are now in a position to state one of our main results in the case of Clifford algebra $\mathbb R_n$ for  slice monogenic extensions to  the open unit ball
$$\mathbb B:=\Big\{x\in\mathbb R^{n+1}: |x|<1\Big\}$$
of univalent  functions on the unit disc $\mathbb D\subset \mathbb C$.

\begin{theorem}\label{main-thm-Clifford} Let  $F: \mathbb D\rightarrow \mathbb C$ be a   univalent function such that $F(0)=0$ and $F'(0)=1$, and let
$f:\mathbb B\rightarrow \mathbb R_n$ be the slice monogenic extension of $F$. Then for each $x\in\mathbb B,$ the following inequalities hold:
\begin{eqnarray}\label{eq:12}
\frac{|x|}{(1+|x|)^2}\leq |f(x)|\leq \frac{|x|}{(1-|x|)^2};
\end{eqnarray}
\begin{eqnarray}
\frac{1-|x|}{(1+|x|)^3}\leq |f'(x)|\leq \frac{1+|x|}{(1-|x|)^3} ;
\end{eqnarray}\label{eq:11}
\begin{eqnarray}\label{eq:12}
\frac{1-|x|}{1+|x|}\leq \big|xf'(x)\ast f^{-\ast}(x)\big|\leq \frac{1+|x|}{1-|x|}.
\end{eqnarray}
Moreover, equality holds for one of these six inequalities at some point $x_0\in \mathbb B\setminus\{0\}$ if and only if
$$f(x)=x(1-xe^{i\theta})^{-\ast2},\qquad\forall \,\,x \in \mathbb B$$
for some $ \theta \in \mathbb R$.
\end{theorem}

Although Theorem \ref{main-thm-Clifford} coincides in form with Theorem \ref{Th:dg-theorem},   the classical approach to Theorem \ref{Th:dg-theorem} can not be directly applied in this new  case of Clifford algebra $\mathbb R_n$, since  there lacks a fruitful theory of compositions for  slice monogenic functions.   We shall reduce Theorem \ref{main-thm-Clifford} to Theorem \ref{Th:dg-theorem} via  a new convex combination  identity (see (\ref{eq: modulus12}) below).  We remark that in contrast to the setting of several complex variables in which the growth and distortion theorems fail to hold in general \cite{Ca} and can only be restricted to the starlike or convex subclasses, our result for  slice monogenic extensions of univalent   functions holds without  extra geometric assumptions.
This new phenomenon is in  a certain sense related to the rigidity of the functions under consideration.  There is a significant difference existing between slice monogenic functions and holomorphic functions of several complex variables, although they are both the generalizations in higher dimensions of holomorphic functions of one complex variable. The former are  closer to holomorphic functions of one complex variable, and each of them can be completely determined by its values on a set that lies in a complex slice and has an accumulation point in its domain of definition. However, this is not the case for the latter, each of which is not always determined by its values on a complex submanifold of positive codimensions in its domain of definition. From this perspective, we realize that  holomorphic functions of several complex variables are less rigid than slice monogenic functions so that certain extra geometric assumptions such as starlikeness and  convexity  are naturally present in the geometric function theory in several complex variables.

A result analogous to  Theorem \ref{main-thm-Clifford} also holds in the setting of quaternions (see  Theorem \ref{th:DG-theorem153} below). As an application, we can prove a covering theorem, that is the so-called Koebe type one-quarter theorem (see Theorem \ref{th:Koebe-theorem} below,  a generalization of \cite[Theorem 3.11 (1)]{GGCO}), with the help of the open mapping theorem, which is by now  known to hold   only for slice regular functions defined on symmetric slice domains in $\mathbb H$ with  values in $\mathbb H$ rather than slice monogenic functions defined on symmetric slice domains  in  paravector space $\mathbb R^{n+1}$ with values in Clifford algebra $\mathbb R_n$.

We now describe in  more detail the structure of the paper. In Section \ref{Preliminaries}, we set up basic notation and give some preliminary results. In Section \ref{GDTs for slice monogenicity}, we first prove in Proposition \ref{eq:C-version} a general formula to express the squared norm of a slice monogenic function defined on a symmetric slice domain in the paravector space $\mathbb R^{n+1}$, in terms of the values of the function at two conjugate points on some fixed slice of the domain. We then provide in Lemma \ref{eq:norm} for slice monogenic functions that preserve one slice the aforementioned convex combination identity, which is the key ingredient of  proving Theorem  \ref{main-thm-Clifford}. Section \ref{GDTs for slice regularity} is devoted to the detailed proofs of the analogous results and the Koebe type one-quarter theorem (Theorem \ref{th:Koebe-theorem}) for slice regular functions in the quaternionic setting. Thanks to the speciality of quaternions, we can also provide in Corollary \ref{eq:modulus relation151} a sufficient and necessary condition under which the aforementioned convex combination identity holds identically. Finally, Section \ref{Concluding remarks} comes a concluding remark and an open question connected with the  subject of the present paper.

\section{Preliminaries}\label{Preliminaries}

We recall in this section some necessary definitions and preliminary results on real Clifford algebras and slice monogenic functions.
To have a more complete insight, we refer the reader to the monograph \cite{Co2}.

The real Clifford algebra $\mathbb R_n=\textrm{Cl}_{0,n}$ is an associative algebra over $\mathbb R$ generated by $n$ basis elements $e_1,e_2,\ldots,e_n$, subject to the relations
$$e_ie_j+e_je_i=-2\delta_{ij},\quad i,j=1,2,\ldots,n.$$
As a real vector space, $\mathbb R_n$ has dimension $2^n$. Each  element $b$ in $\mathbb R_n$ can be represented \textit{uniquely} as $$b=\sum\limits_{A}b_Ae_A, $$
where  $b_A \in \mathbb{R}$, $e_{0}=1$, $e_{A}:=e_{h_1}e_{h_2}\ldots e_{h_r}$, and
 $A=h_1\ldots h_r$ is a multi-index such that $1\leq h_{1}<\cdots<h_{r}\leq n$.
The real number $b_0$ is called the \textit{scalar} part of $b$ and is denoted by ${\rm{Sc}}(b)$ as usual. The \textit{Clifford conjugate} of each generator $e_i$, $i=1, 2,\ldots, n,$ is defined to be $\bar{e}_i=-e_i$, and thus extends to each $e_A$ by setting  $\bar{e}_A:=\bar{e}_{h_{r}}\bar{e}_{h_{r-1}}\ldots \bar{e}_{h_{1}}=(-1)^re_{h_{r}}e_{h_{r-1}}\ldots e_{h_{1}}=(-1)^{r(r+1)/2}e_A,$ and further extends by linearity to each element $b=\sum_{A}b_Ae_A\in \mathbb R_n$ so that
 $$\bar{b}=\sum\limits_{A}b_A\bar{e}_A.$$ Therefore, the Clifford  conjugate is an anti-automorphism of $\mathbb R_n$, i.e. $\overline{ab}=\bar{b}\bar{a}$ for any $a, b\in \mathbb R_n$.
 Moreover, the Euclidean inner product on $\mathbb R_n\simeq\mathbb R^{2^n}$ is given by
\begin{equation}\label{inner-product-definition001}
\langle a,b\rangle:={\rm{Sc}}(a\bar{b})=\sum\limits_{A}a_Ab_A
\end{equation}
for any $a=\sum\limits_{A}a_Ae_A$, $b=\sum\limits_{A}b_Ae_A\in \mathbb R_n$, then it follows from the simple identity
$$\langle a,b\rangle=\frac12\big(|a+b|^2-|a|^2-|b|^2\big)$$
that
\begin{equation}\label{99}
\langle a,b\rangle=\langle b,a\rangle=\langle \bar{a},\bar{b}\rangle=\langle \bar{b},\bar{a}\rangle.
\end{equation}
It is worth remarking here that for $\mathbb R_n(n\geq 3)$ the multiplicative property of the Euclidean norm fails in general, and holds only for some special cases (see \cite[Proposition 2.1.17]{Co2} or \cite[Theorem 3.14 (ii)]{GHS}. In particular, it holds that
\begin{equation}\label{mult-norm}
|ab|=|ba|=|a||b|
\end{equation}
whenever one of $a$ and $b$ is a paravector (see below for this definition).
 This simple fact will be useful for our argument in Section \ref{GDTs for slice monogenicity}.

For convenience, some specific elements in $\mathbb R_n$ can be identified with vectors in the
Euclidean space $\mathbb R^{n+1}$: an element $(x_1, x_2,\ldots, x_n)\in \mathbb R^n$ will be identified with a so-called \textit{1-vector} in the Clifford algebra $\mathbb R_n$ through the map $(x_1, x_2,\ldots, x_n)\mapsto \underline{x}=x_1e_1+e_2x_2+\cdots+x_ne_n$; an element $(x_0, x_1,\ldots, x_n)\in \mathbb R^{n+1}$ will be identified with $x=x_0+\underline{x}=x_0+x_1e_1+\cdots+x_ne_n$, which is called a \textit{paravector}. Now for any  two 1-vectors $x$, $y\in \mathbb R^{n}$, the Euclidean inner product becomes
$$\langle x, y\rangle
={\rm{Sc}}(x\bar{y})
=-\frac12(xy+yx),$$
and consequently,
$$xy
=-\langle x, y\rangle+x\wedge y,$$
where
$$x\wedge y:=\frac12(xy-yx)$$
 is called the  \textit{outer product} (cf. \cite[p.4]{BDS}, \cite[p.58]{GHS}) or \textit{wedge product} (cf. \cite[p.218, Definition 4.1.9]{CSSS}, \cite[p.21]{Co2}) of $x$ and $y$. It is noteworthy here that \textit{in general the operator $\wedge$ is a mapping from $\mathbb R^n\times\mathbb R^n$ to $\mathbb R_n$, not to $\mathbb R^n$}.
Furthermore, under the identifications above, a vector $x$ in $\mathbb R^{n+1}$ can be taken as a  Clifford number
 $$x=x_0+\sum_{i=1}^{n}x_ie_i$$
so that it has inverse
$$x^{-1}=\frac{\bar x}{|x|^2},$$
where $\bar x$ is the \emph{conjugate} of $x$ given by  $\bar{x}=x_0-\sum_{i=1}^{n}x_ie_i$ and the  norm of $x$ is induced by the inner product given above, i.e. $|x|=\langle x,x\rangle^{\frac12}$. Every $x=x_0+x_1e_1+\cdots+x_ne_n\in\mathbb R^{n+1}$ is composed by the \textit{scalar} part ${\rm{Sc}}(x)=x_0\in\mathbb R$ and the \textit{vector} part $\underline{x}=x_1e_1+\cdots+x_ne_n\in\mathbb R^{n}$, and  it can be expressed alternatively  as $x=u+Iv$, where $u$, $v\in \mathbb R$ and
$$I=\frac {\underline{x}}{|\underline{x}|}$$
if $\underline{x}\neq 0$, otherwise  we take $I$ arbitrarily in $\mathbb R^n$ such that $I^2=-1$.
Then $I $ is an element of the unit $(n-1)$-sphere of   $1$-vectors in $\mathbb R^n$,
$$\mathbb S=\Big\{\underline{x}=x_1e_1+\cdots+x_ne_n\in\mathbb R^{n}:x_1^2+\cdots+x_n^2=1\Big\}.$$
For every $I \in \mathbb S $ we will denote by $\mathbb C_I$ the plane $ \mathbb R \oplus I\mathbb R $, isomorphic to $ \mathbb C$, and, if $U \subseteq \mathbb R^{n+1}$, by $U_I$ the intersection $ U \cap \mathbb C_I $. Also, for $R>0$, we will denote the open ball of $\mathbb R^{n+1}$ centred at the origin with radius $R$ by
$$B(0,R)=\Big\{x \in \mathbb R^{n+1}:|x|<R\Big\}.$$

We can now recall the definition of slice monogenicity.
\begin{definition} \label{de: regular} Let $U$ be a domain in $\mathbb R^{n+1}$. A function $f :U \rightarrow \mathbb R_n$ is called \emph{slice monogenic} if, for all $ I \in \mathbb S$, its restriction $f_I$ to $U_I$ is \emph{holomorphic}, i.e.  it has continuous partial derivatives and satisfies
$$\bar{\partial}_I f(u+vI):=\frac{1}{2}\left(\frac{\partial}{\partial u}+I\frac{\partial}{\partial v}\right)f_I (u+vI)=0$$
for all $u+vI\in U_I $.
 \end{definition}

The natural domains of definition in the theory of slice monogenic functions are  symmetric slice domains.

\begin{definition} \label{de: domain}
Let $U$ be a domain in $\mathbb R^{n+1}$.
\begin{enumerate}[leftmargin=1.7pc, parsep=4pt,label=(\roman*)]
\item $U$ is called a \textit{slice domain}  if it intersects the real axis and if  for each $I \in \mathbb S $, $U_I$  is a domain in $ \mathbb C_I $.

\item  $U$ is called an \textit{axially symmetric domain} if for every point  $u + vI \in U$, with  $u,v \in \mathbb R $ and $I\in \mathbb S$, the entire sphere $u + v\mathbb S$ is contained in $U$.
\end{enumerate}
\end{definition}

A domain in $\mathbb R^{n+1}$ is called a \textit{symmetric slice domain} if it is not only a slice domain, but also an axially symmetric domain. By the very definition, an open ball $B(0,R)$ is  a typical  symmetric slice domain. From now on, we will focus mainly on slice monogenic functions on $B(0,R)$.
In most cases, the following results hold, with appropriate changes, for symmetric slice domains more general than open balls of the type $B(0,R)$.
For slice monogenic functions a natural definition of derivative is given by the following.
\begin{definition} \label{de: derivative}
Let $f :B(0,R) \rightarrow \mathbb R_n$  be a slice monogenic  function. The \emph{slice derivative} of $f$ is defined to be
$$\partial_I f(u+vI):=\frac{1}{2}\left(\frac{\partial}{\partial u}-I\frac{\partial}{\partial v}\right)f_I (u+vI).$$
 \end{definition}
Notice that the operators $\partial_I$ and $\bar{\partial}_I $ commute, and
$$\partial_I f(u+vI)=\frac{\partial}{\partial u}f(u+vI)$$ holds for slice monogenic functions. Therefore, the
slice derivative of a slice monogenic function is still slice monogenic  so that we can iterate the differentiation to obtain the $k$-th
slice derivative
$$\partial^{k}_I f(u+vI)=\bigg(\frac{\partial}{\partial u}\bigg)^{k} f(u+vI),\quad\,\forall \,\, k\in \mathbb N. $$
In what follows, for the sake of simplicity, we will direct denote the $k$-th slice derivative $\partial^{k}_I f$ by $f^{(k)}$ for every $k\in \mathbb N$.

As shown in \cite{Co}, a paravector power series $\sum_{k=0}^{\infty}x^k a_k$ with $\{a_k\}_{k \in \mathbb N} \subset \mathbb R_n$ defines a slice monogenic function in its domain of convergence, which proves to be an open ball $B(0,R)$ with $R$ equal to the radius of convergence of the power series. The converse result is also true.
\begin{theorem}\label{eq:Taylor}
A function f is slice monogenic on $B = B(0,R) $ if and only if $f$ has a power series expansion
$$f(x)=\sum\limits_{k=0}^{\infty}x^k a_k\quad with \quad a_k=\frac{f^{(k)}(0)}{k!}.$$
\end{theorem}

A fundamental result in the theory of slice monogenic functions is described by the splitting lemma, which relates the notion of  slice monogenicity to the classical notion of holomorphicity (see \cite{Co}).
\begin{lemma}\label{eq:Splitting}
Let $f$ be a slice monogenic function on $B = B(0,R)$. For each $I_1=I\in \mathbb S$, let $I_2,\dots,I_n$ be a completion to a basis of $\mathbb R^n$ satisfying the defining relations $I_iI_j+I_jI_i=-2\delta_{ij}$. Then there exist $2^{n-1}$ holomorphic functions $F_A:B_I\rightarrow \mathbb C_I$ such that for every $z=u+vI\in B_I $,
$$f_I(z)=\sum\limits_{|A|=0}^{n-1}F_A(z)I_A,$$
where $I_{A}=I_{i_{1}}I_{i_{2}}\ldots I_{i_{r}}$, $A=i_1i_2\ldots i_r$ is a multi-index such that $2\leq i_{1}<\cdots<i_{r}\leq n$ when $r>0$, or $I_0=1$ when $r=0$.
\end{lemma}

The following version of the identity principle is one of direct consequences of the preceding lemma (see \cite{Co}).

\begin{theorem}\label{th:IP-theorem}
Let $f$ be a slice monogenic function on $B = B(0,R)$. Denote by $\mathcal{Z}_f$ the zero set of $f$, $$\mathcal{Z}_f=\Big\{x\in B:f(x)=0\Big\}.$$
If there exists an $I\in \mathbb S$ such that $B_I\cap \mathcal{Z}_f$ has an accumulation point in $B_I$, then $f$ vanishes identically on $B$.
\end{theorem}

Another useful result is the following (see \cite{CS}).

\begin{theorem}\label{eq:formula}
Let $ f $ be a slice monogenic function on a symmetric slice domain $U \subseteq \mathbb R^{n+1} $ and let $I \in \mathbb S.$ Then for all $ u+vJ\in U $ with $J \in \mathbb S$, the following equality holds
$$ f(u+vJ)=\frac{1}{2}\Big(f(u+vI)+f(u-vI)\Big)+\frac{1}{2}JI\Big(f(u-vI)-f(u+vI)\Big).$$
In particular, for each sphere of the form $u+v\mathbb S$ contained in $U$, there exist $b$, $c \in \mathbb R_n$ such that $f(u+vI)=b+Ic$ for all $I\in\mathbb S$.
\end{theorem}

Thanks to this result, it is possible to recover the values of a slice monogenic function on symmetric slice domains, which are more
general  than open balls centred at the origin, from its values on a single
slice. This yields an extension theorem  that in the special case of functions that are slice monogenic on $B(0,R)$ can be obtained by means of their power series
expansions.

\begin{remark}\label{R-extension}
 Fix  an element $I\in\mathbb S$ and we denote by  $B_I$ the intersection $B(0,R)\cap \mathbb C_I$ of the open ball $B(0,R)$ with the complex plane $\mathbb C_I$. Given  a holomorphic function $f_I:B_I\rightarrow\mathbb C_I$ with the
power series expansion taking the form
$$f_I(z)=\sum\limits_{k=0}^{\infty}z^ka_k,$$
where $\{a_k\}_{k\in \mathbb N}\subset\mathbb C_I$, then the unique slice monogenic extension of $f_I$ to the whole ball $B(0,R)$ is the function
given by
$$f(x):={\rm{ext}}(f_I)(x)=\sum\limits_{k=0}^{\infty}x^ka_k,$$
which takes values in $\mathbb R_n$.
The uniqueness is guaranteed by the identity principle (Theorem \ref{th:IP-theorem}). In Section \ref{GDTs for slice monogenicity}, we will  establish the growth and distortion theorems  for such a class of slice monogenic functions  that are injective on $B_I$.
\end{remark}

Since slice monogenicity does not keep under the usual pointwise product of two slice monogenic functions, a new multiplication operation, called  slice monogenic product (or $\ast$-product), appears via  a suitable modification of the usual one subject to noncommutative setting, and plays a key role in the theory of slice monogenic functions. On open balls centred at the origin, the slice monogenic product of two slice monogenic functions is defined by means of their power series expansions (see \cite{Co1, Co2}).

\begin{definition}\label{R-product}
Let $f$, $g:B=B(0,R)\rightarrow \mathbb R_n$ be two slice monogenic functions and let
$$f(x)=\sum\limits_{k=0}^{\infty}x^ka_k,\qquad g(x)=\sum\limits_{k=0}^{\infty}x^kb_k$$
be their power series expansions. The \textit{slice monogenic product} (\textit{$\ast$-product}) of $f$ and $g$ is the function defined by
$$f\ast g(x)=\sum\limits_{k=0}^{\infty}x^k\bigg(\sum\limits_{j=0}^k a_jb_{k-j}\bigg),$$
which is slice monogenic on $B$.
\end{definition}
We now recall more definitions  (see e.g. \cite{Co1, Co2, Ghiloni2011, Ghiloni1}).
\begin{definition} \label{de: R-conjugate}
Let $ f(x)=\sum_{k=0}^{\infty}x^ka_k $ be a slice monogenic function on $B=B(0,R)$. We define the \emph{slice monogenic conjugate} of $f$ as
$$f^c(x)=\sum\limits_{k=0}^{\infty}x^k\bar{a}_k,$$
and  the \emph{symmetrization} of $f$ as
\begin{equation}\label{Definition of symmetrization for SM functions}
f^s(x):=\sum\limits_{k=0}^{\infty}x^k{\rm{Sc}}\bigg(\sum\limits_{j=0}^k a_j\bar{a}_{k-j}\bigg).
\end{equation}
Moreover, we define the \textit{normal function}  of $f$ as
\begin{equation}\label{Definition of normal-function for SM functions}
N(f)(x):=f\ast f^c(x)=\sum\limits_{k=0}^{\infty}x^k\bigg(\sum\limits_{j=0}^k a_j\bar{a}_{k-j}\bigg).
\end{equation}
These three functions are slice monogenic  on $B$.
\end{definition}

\begin{remark} \label{remarks on SM-conjugate}
Several useful remarks concerning Definitions \ref{R-product}  and \ref{de: R-conjugate} are in order:
\begin{enumerate}[leftmargin=1.7pc, parsep=4pt,label=(\roman*)]
\item [{\rm{(i)}}] Slice monogenic product ($\ast$-product), slice monogenic conjugate,  and symmetrization can also be defined for slice monogenic functions $f$ on symmetric slice domains $U$ in $\mathbb R^{n+1}$ (we refer the interested reader to \cite{Co1} or \cite[Section 2.6]{Co2} for  details). Moreover, for any two slice monogenic functions  $f, g:U\to\mathbb R_n$ and each point $x_0\in \mathbb R$, we can define two slice monogenic functions $f_{x_0}$ and $g_{x_0}$ on the symmetric slice domain $U_{x_0}:=U-x_0$ by setting
    $$f_{x_0}(x)=f(x+x_0),\qquad g_{x_0}(x)=g(x+x_0)$$ for each $x\in U_{x_0}$. Then we have the following identity
    $$(f\ast g)_{x_0}=f_{x_0}\ast g_{x_0}.$$
    This follows from the identity principle together  with the fact that when restricted to the real axis, slice monogenic product is just the usual pointwise one.

\item [{\rm{(ii)}}]
For slice monogenic functions on open balls of type $B:=B(0, R)$, the notion of slice monogenic conjugate  coincides with the one introduced in \cite[Definition 5.4]{Co1} (see \cite[Proposition 5.5]{Co1}). The notion of symmetrization given here is also equivalent to the one introduced in \cite[Definition 5.6]{Co1}. To see this, we proceed as follows. For a slice monogenic function $f:B\to \mathbb R_n$, we denote by $f^{\textbf{s}}$ the symmetrization of $f$ according to \cite[Definition 5.6]{Co1}.
By considering  the power series expansion of $f^{\textbf{s}}$, we may assume that
\begin{equation}\label{Sym-Sabadini}
f^{\textbf{s}}(x)=\sum_{k=0}^{\infty}x^k\alpha_k.
\end{equation}
We also fix an element $I\in\mathbb S$. Then according to  \cite[p.386]{Co1} or \cite[p.50]{Co2}, for each $x\in B_I$, we have
$$f^{\textbf{s}}(x)={\rm{Sc}}\big(f\ast f^c(x)\big)+\big\langle f\ast f^c(x), \, I\big\rangle I.$$
Now substituting $(\ref{Definition of normal-function for SM functions})$ and $(\ref{Sym-Sabadini})$ into the preceding equality, we see that for each $x\in B\cap\mathbb R$,
$$\sum_{k=0}^{\infty}x^k\alpha_k
=\sum_{k=0}^{\infty} x^k{\rm{Sc}}\Big(\sum_{j=0}^k a_j\bar{a}_{k-j}\Big)
+\sum_{k=0}^{\infty}x^k \Big\langle \sum_{j=0}^k a_j\bar{a}_{k-j}, \, I\Big\rangle I.$$
For each $k\in \mathbb N$, since $\sum_{j=0}^k a_j\bar{a}_{k-j}$ is invariant under Clifford conjugate (see $(\ref{Coefficient conjugate invariant})$ below), the second summation on the right-hand side of the preceding equality must vanish identically. Indeed, in view of $(\ref{99})$,
$$\Big\langle \sum_{j=0}^k a_j\bar{a}_{k-j}, \, I\Big\rangle
=\Big\langle \sum_{j=0}^k \overline{a_j\bar{a}_{k-j}}, \, \overline{I}\Big\rangle
=-\Big\langle \sum_{j=0}^k a_j\bar{a}_{k-j}, \, I\Big\rangle,$$
which must be zero. Consequently, we deduce that the following equality
$$\sum_{k=0}^{\infty}x^k\alpha_k
=\sum_{k=0}^{\infty} x^k{\rm{Sc}}\Big(\sum_{j=0}^k a_j\bar{a}_{k-j}\Big)$$
holds for all $x\in B\cap\mathbb R$. By uniqueness,
$$\alpha_k={\rm{Sc}}\Big(\sum_{j=0}^k a_j\bar{a}_{k-j}\Big),\qquad \forall \, k\in \mathbb N.$$
This shows that $f^{\textbf{s}}$ is  the same as $f^s$ defined in $(\ref{Definition of symmetrization for SM functions})$.

\item [{\rm{(iii)}}] In view of ${\rm{(i)}}$,  the definition $N(f):=f\ast f^c$ is also valid  for slice monogenic functions $f$ on symmetric slice domains in $\mathbb R^{n+1}$.

\item [{\rm{(iv)}}] The notation $N(f)$ in the definition of normal functions is chosen in accordance with \cite[Definition 11]{Ghiloni1}, which treated the case of slice functions on symmetric open subsets of the so-called \textit{quadratic cone} of a finite-dimensional real alternative $^{\ast}$-algebra.
\item [{\rm{(v)}}]  For each slice monogenic function $f$ on a symmetric slice domain  $U\subseteq\mathbb R^{n+1}$ and each element $I\in\mathbb S$, the restriction $N(f)_I$ of $N(f)$ to $U_I:=U\cap \mathbb C_I$ coincides with the function $f_I\ast f^c_I: U_I\to \mathbb R_n$ considered in \cite{Co1} or \cite[Section 2.6]{Co2}.
\end{enumerate}
\end{remark}

With parts {\rm{(i)}} and  {\rm{(iii)}} of  Remark \ref{remarks on SM-conjugate} in mind,
the inverse element of a non-identically vanishing slice monogenic functions   with respect to the $\ast$-product can be defined under a suitable condition.

\begin{definition} \label{de: R-Inverse}
Let $f$ be a slice monogenic function on a symmetric slice domain $U\subseteq \mathbb R^{n+1}$ such that
 $$N(f)(U_I)\subseteq\mathbb C_I$$
for \textit{some} $I\in \mathbb S$. If $f$ does not vanish identically, its \emph{slice monogenic inverse} is the function defined by
$$f^{-\ast}(x):=f^s(x)^{-1}f^c(x),$$
which is
slice monogenic on $U \setminus \mathcal{Z}_{f^s}$. Here $\mathcal{Z}_{f^s}$ denotes the zero set of the symmetrization $f^s$ of $f$.
\end{definition}

\begin{remark}\label{remarks on SM-Inverse}
Two useful remarks concerning Definition  \ref{de: R-Inverse} are in order:
\begin{enumerate}[leftmargin=1.7pc, parsep=4pt,label=(\roman*)]
\item [{\rm{(i)}}] For each function $f$ as described in Definition \ref{de: R-Inverse}, the  requirement that $$N(f)(U_I)\subseteq\mathbb C_I$$  for some  $I\in\mathbb S$  is designed to guarantee that $f^s_I$ coincides with $N(f)_I=f_I\ast f^c_I$,
    see \cite[Definition 2.6.10]{Co2}, although this fact is not explicitly proven in \cite{Co2}.

\item [{\rm{(ii)}}] Also we will see, in the proof of Proposition
\ref{Proof of Inverse verification} below, that for each function $f$ as described in Definition \ref{de: R-Inverse}, the coefficients appeared  in $(\ref{Definition of normal-function for SM functions})$ are \textit{real numbers}. This implies that for each such function $f$, its normal function $N(f)$ is the same as its symmetrization $f^s$, which is a slice preserving function  so that its slice monogenic inverse
     $$f^{-\ast}(x)=f^s(x)^{-1}f^c(x)=\big(N(f)(x)\big)^{-1}f^c(x)$$
     is indeed slice monogenic on $U\setminus \mathcal{Z}_{f^s}$. Furthermore, it is well worth noting that in view of \cite[Remark 2.6.8 and Lemma 2.5.12]{Co2}, the zero set $\mathcal{Z}_{f^s}$ of $f^s$ is precisely the union of isolated spheres of the form $u+v\mathbb S$ with $u$, $v\in \mathbb R$. This implies  that $U\setminus \mathcal{Z}_{f^s}$ is a symmetric slice domain in $\mathbb R^{n+1}$.
\end{enumerate}
\end{remark}

The function $f^{-\ast}$ defined in Definition \ref{de: R-Inverse} deserves the name of slice monogenic inverse of $f$ due to the following:
\begin{proposition}\label{Proof of Inverse verification}
Let $f$ be as described in Definition  $\ref{de: R-Inverse}$. Then  we have
\begin{equation}\label{Inverse verification}
\left.f\right|_{U \setminus \mathcal{Z}_{f^s}}\ast f^{-\ast}
=f^{-\ast}\ast \left.f\right|_{U \setminus \mathcal{Z}_{f^s}}=1,
\end{equation}
and
\begin{equation}\label{Double Inverse}
(f^{-\ast})^{-\ast}=\left. f\right|_{U \setminus \mathcal{Z}_{f^s}}.
\end{equation}
\end{proposition}

This proposition is quite important in the theory of slice monogenic functions. Equalities in $(\ref{Inverse verification})$ firstly appeared in \cite[Proposition 5.9]{Co1}, while the proofs given there and in \cite[Proposition 2.6.11]{Co2} seem incomplete. Indeed, the equality $f_I\ast f_I^c=f_I^c\ast f_I$ (which is equivalent to $N(f)=N(f^c)$,  in view of  Remark \ref{remarks on SM-conjugate} (v) and the identity principle) is used without proving it. A different approach has been used in \cite[Proposition 3.2]{Co2011}. A complete treatment has been given in \cite[Section 2]{Ghiloni2} in the case of slice functions, which subsumes the case of slice monogenic functions. In order to make our presentation self-contained, we provide here a detailed proof of Proposition \ref{Proof of Inverse verification}.

\begin{proof}
We first prove equality $(\ref{Inverse verification})$. To this end, we need the following well known facts:
\begin{description}
  \item[Fact 1 ] \textit{For  any $a$, $b\in \mathbb R_n$, $ab=1$ if and only if $ba=1$.}
\end{description}
\begin{description}
  \item[Fact 2 ] \textit{For  each  $a\in \mathbb R_n$, $a\overline{a}=0$ if and only if $a=0$.}
\end{description}
Indeed, Fact $1$   holds  for all finite-dimensional associative algebras (see e.g. \cite[Theorem 1.2.1]{Drozd-Kirichenko}), and Fact $2$, which immediately follows from $(\ref{inner-product-definition001})$, is called \textit{non-singularity}  of $\mathbb R_n$.

Note that $f$  does not vanish identically on $U$, and so does the restriction $\left. f\right|_{U\cap\mathbb R}$ of $f$ to  $U\cap\mathbb R$, in view of the identity principle. Thus we can find one point $x_0\in U\cap\mathbb R$ and a positive number $R>0$ such that the open ball $B(x_0, R)$ is contained in $U$ and $f$ is  nowhere vanishing on $B(x_0, R)$. Thanks to Remark  \ref{remarks on SM-conjugate} (i), we may further assume that $x_0=0$ without loss of generality.
Now we expand $f$ on $B:=B(0, R)$ as
 $$f(x)=\sum_{k=0}^{\infty}x^ka_k.$$
Since there exists an element $I\in\mathbb S$ such that $N(f)=f\ast f^c$ maps $U_I$ into $\mathbb C_I$ (also maps $B_I$ into $\mathbb C_I$), and
\begin{equation}\label{Coefficient conjugate invariant}
\sum_{j=0}^k \overline{a_j\bar{a}_{k-j}}=\sum\limits_{j=0}^k a_{k-j}\bar{a}_j \xlongequal[\quad]{j\rightarrow k-j}\sum\limits_{j=0}^k a_j\bar{a}_{k-j},
\end{equation}
we see that for each $k\in\mathbb N$, $\sum_{j=0}^k a_j\bar{a}_{k-j}$ must be a real number. Therefore, $f\ast f^c$ is slice preserving and maps $ B\cap\mathbb R$ into $\mathbb R$. We next show that
\begin{equation}\label{Key relation in  the proof of R-inverse}
f^c\ast f=f\ast f^c.
\end{equation}
We proceed as follows. In view of Definition \ref{de: R-conjugate},
$$\left. f\ast f^c\right|_{B\cap\mathbb R}=\left.(f \bar{f}\,)\right|_{B\cap\mathbb R}.$$
Since $f\ast f^c(B\cap\mathbb R)\subseteq\mathbb R$, we deduce that
the restriction $\left.(f \bar{f}\,)\right|_{B\cap\mathbb R}$ takes values in $\mathbb R$ as well. This together with the preceding two facts implies that
$$\left.(f \bar{f}\,)\right|_{B\cap\mathbb R}=\left.(\bar{f}f)\right|_{B\cap\mathbb R}.$$
The right-hand side is no other than the restriction $\left.f^c\ast f\right|_{\mathbb B\cap\mathbb R}$, according to  Definitions \ref{R-product} and \ref{de: R-conjugate}. Now we obtain that
$f\ast f^c$ coincides with  $f^c\ast f$ on $B\cap\mathbb R\subset U$, and hence on $U$ by the identity principle.
Now by using \cite[Proposition 2.6.9]{Co2}, Remark \ref{remarks on SM-Inverse} (ii) and equality (\ref{Key relation in  the proof of R-inverse}),  we can conclude the proof of equality $(\ref{Inverse verification})$ as follows:
$$f^{-\ast}\ast f=\frac1{f^s}(f^c\ast f)=\frac1{f^s}(f\ast f^c)=\frac1{f^s}N(f)=1,$$
and
$$f\ast f^{-\ast}=f\ast(\frac1{f^s}f^c)=\frac1{f^s}(f\ast f^c)=1.$$

Now it remains to prove $(\ref{Double Inverse})$.  In view of the very definition, we first need to show that $f^{-\ast}$ satisfies the condition given in Definition \ref{de: R-Inverse}. To see this, let $I$ be an element of $\mathbb S$ such that $f$
satisfies the assumption. From the above argument, we know that $f^s=N(f)$ is slice preserving. This together with   $(\ref{Key relation in  the proof of R-inverse})$ and \cite[Proposition 2.6.9]{Co2} implies that
$$f^{-\ast}\ast (f^{-\ast})^c=\frac1{N(f)}$$
so that $f^{-\ast}$ satisfies the assumption given in Definition \ref{de: R-Inverse} and hence $(f^{-\ast})^{-\ast}$ is well-defined on $U\setminus \mathcal{Z}_{f^s}$.  Now $(\ref{Double Inverse})$ follows from $(\ref{Inverse verification})$ and uniqueness of $(f^{-\ast})^{-\ast}$.
\end{proof}

\section{Growth and Distortion Theorems for  Slice Monogenic Functions}\label{GDTs for slice monogenicity}
In this section, we establish in the setting of Clifford algebra $\mathbb R_n$ the growth  and distortion theorems for  slice monogenic extensions to  the open unit ball $\mathbb B:=\big\{x\in\mathbb R^{n+1}: |x|<1\big\}$ of univalent   functions on the unit disc $\mathbb D\subset \mathbb C$.  We begin with a  technical proposition.
To present  it more generally, we will digress for a moment to slice monogenic functions on general symmetric  slice domains.

\begin{proposition}\label{eq:C-version}
Let $U \subseteq \mathbb R^{n+1}$ be a symmetric slice domain and  $f:U\rightarrow \mathbb R_n$ a slice monogenic  function. Then  for every $x=u+vJ \in U $ and every $I\in \mathbb S$,  there   holds the identity
\begin{equation}\label{eq:C-relation}
\big|f(x)\big|^2=\frac{1+\langle I,J\rangle}{2}\big|f(y)\big|^2+
\frac{1-\langle I,J\rangle}{2}\big|f(\bar y)\big|^2-
 \Big\langle f(y)\overline{f(\bar y)}, \, I\wedge J\Big\rangle,
\end{equation}
where  $y=u+vI$ and $\bar y=u-vI$.
\end{proposition}

\begin{proof}
Fix an arbitrary point  $ x=u+vJ \in U$ and an element $I\in \mathbb S$. Set $y:=u+vI$ and $\bar y:=u-vI$. It follows from Theorem \ref{eq:formula} that
\begin{equation}\label{eq: Rep99}
f(x)=\frac{1}{2}\big(f(y)+f(\bar y)\big)-\frac{1}{2}JI\big(f(y)-f(\bar y)\big)
\end{equation}
Notice that, in vector notation,
\begin{equation}\label{eq: vector product101}
\langle I,J\rangle={\rm{Sc}}(I\bar J)=-\frac{1}{2}(IJ+JI),
\end{equation}
and
\begin{equation}\label{1010}
I\wedge J=\frac{1}{2}(IJ-JI).
\end{equation}
We shall use the simple identity that
\begin{equation}\label{eq: Rep100} |a+b|^2=|a|^2+|b|^2+2\langle a,b\rangle
\end{equation}
for any $a$, $b\in \mathbb R_n\simeq\mathbb R^{2^n}$.

 Observe that $I$ and $J$ are 1-vectors and hence are paravectors. In view of (\ref{mult-norm}), it holds that
$$\big|JI\big(f(y)-f(\bar y)\big)\big|=\big|f(y)-f(\bar y)\big|. $$
Taking modulus on both sides of (\ref{eq: Rep99}) and applying (\ref{eq: Rep100})  to obtain
\begin{eqnarray}\label{eq: modulus101}
 \begin{split}
\big|f(x)\big|^2
=& \frac{1}{4}\Big(\big|f(y)+f(\bar y)\big|^2+\big|f(y)-f(\bar y)\big|^2\Big)-
\\
&
\frac12\Big\langle f(y)+f(\bar y),\, JI\big(f(y)-f(\bar y)\big)\Big\rangle\\
=&: A-\frac{1}{2}B.
\end{split}
\end{eqnarray}
Again applying (\ref{eq: Rep100}), it is evident that
\begin{equation}\label{102}
 A=\frac{1}{2}(|f(y)|^2+|f(\bar y)|^2).
\end{equation}
To calculate the term  $B$, it first follows from the very definition of inner product (see (\ref{inner-product-definition001})) that
\begin{equation}\label{103}
B=\Big\langle \big(f(y)+f(\bar y)\big)\big(\overline{f(y)}-\overline{f(\bar y)}\,\big),\, JI\Big\rangle
=:B_1+B_2,
\end{equation}
where $B_1=\Big\langle f(y)\overline{f(y)}-f(\bar y)\overline{f(\bar y)},\, JI\Big\rangle$,
and $B_2=\Big\langle f(\bar y)\overline{f(y)}-f(y)\overline{f(\bar y)},\, JI\Big\rangle$.

We next claim that
\begin{equation}\label{104}
B_1=-\langle I,J\rangle\big(|f(y)|^2-|f(\bar y)|^2\big),
\end{equation}
and
\begin{equation}\label{105}
B_2=2 \Big\langle f(y)\overline{f(\bar y)},\,  I\wedge J\Big\rangle.
\end{equation}
Indeed, applying the fact that $\langle a,b\rangle=\langle \bar a,\bar b\rangle$ from (\ref{99}) to $B_1$ yields that
\begin{equation*}\label{106}
  B_1=\Big\langle f(y)\overline{f(y)}-f(\bar y)\overline{f(\bar y)},\, IJ\Big\rangle
\end{equation*}
Combining this, (\ref{eq: vector product101}) and the initial notion of $B_1$, we thus obtain
\begin{equation*}\label{107}
 \begin{split}
   B_1=&\frac12\Big\langle f(y)\overline{f(y)}-f(\bar y)\overline{f(\bar y)},\, IJ+JI\Big\rangle \\
    =&-\Big\langle f(y)\overline{f(y)}-f(\bar y)\overline{f(\bar y)},\, \langle I,J\rangle \Big\rangle\\
    =&-\langle I,J\rangle\Big\langle f(y)\overline{f(y)}-f(\bar y)\overline{f(\bar y)},\, 1\Big\rangle\\
    =&-\langle I,J\rangle\big(|f(y)|^2-|f(\bar y)|^2\big).
 \end{split}
\end{equation*}
Similarly,
\begin{equation*}
B_2=\Big\langle \overline{f(\bar y)\overline{f(y)}},\,
\overline{JI}\Big\rangle-\Big\langle f(y)\overline{f(\bar y)}, \,  JI\Big\rangle=2\Big\langle f(y)\overline{f(\bar y)},\,  I\wedge J\Big\rangle
\end{equation*}
as desired. In the second equality we have used (\ref{1010}). Now substituting (\ref{102})--(\ref{105}) into (\ref{eq: modulus101}) yields that
$$\big|f(x)\big|^2=\frac{1+\langle I,J\rangle}{2}\big|f(y)\big|^2+
\frac{1-\langle I,J\rangle}{2}\big|f(\bar y)\big|^2-
\Big\langle f(y)\overline{f(\bar y)}, \, I\wedge J\Big\rangle,$$
which completes the proof.
\end{proof}

The preceding theorem shows that when $f$ preserves at least one slice, the squared norm of $f$ can thus be expressed as a convex  combination of those in the preserved slice.
\begin{lemma}\label{eq:norm}
Let $f$ be a slice monogenic function on a symmetric slice domain  $U \subseteq \mathbb R^{n+1} $ such that $f(U_I)\subseteq \mathbb C_I $ for some $I\in \mathbb S $.  Then
the convex combination identity
\begin{equation}\label{eq: modulus12}
\big|f(u+vJ)\big|^2 =\frac{1+\langle I,J\rangle}{2}\big|f(u+vI)\big|^2+
\frac{1-\langle I,J\rangle}{2}\big|f(u-vI)\big|^2
\end{equation}
holds for every $ u+vJ \in U $.
\end{lemma}
\begin{proof}
As mentioned before, this lemma is a direct consequence of the preceding proposition. But here, we would like to provide an alternative easier approach to it under having no idea about  Proposition \ref{eq:C-version}.

First, we have the following simple fact, which can be easily verified.
\begin{description}
  \item[Fact]
 \textit{For any $I$, $J\in \mathbb S$, the set
$$\Big\{1,\, I,\,  I\wedge J,\, I(I\wedge J)\Big\}$$
is an orthogonal set of $\mathbb R_n\simeq\mathbb R^{2^n}$.}
\end{description}

As in the preceding proposition, it follows from Theorem \ref{eq:formula} that
\begin{equation}\label{eq: Rep12}
f(x)=\frac{1}{2}\big(f(y)+f(\bar y)\big)-\frac{1}{2}JI\big(f(y)-f(\bar y)\big)
\end{equation}
for every  $ x=u+vJ \in U$ with $y=u+vI$ and $\bar y=u-vI$.
We can rewrite (\ref{eq: Rep12}), in  term of the relation that $$JI=-\langle I,J\rangle+J\wedge I,$$
as
\begin{equation*}
\begin{split}
f(x)&=\frac{1}{2}\Big(\big(1+\langle I,J\rangle \big)f(y)+\big(1-\langle I,J\rangle \big)f(\bar y)\Big)+
\frac{1}{2}\big(J\wedge I\big)\Big(f(\bar y)-f(y)\Big)
\\
&=: \frac{1}{2} A+
\frac{1}{2}(J\wedge I) B
\end{split}
\end{equation*}
By assumption $f(U_I)\subseteq \mathbb C_I $, we thus have
$$A\in \mathbb C_I, \qquad B\in \mathbb C_I.$$
From the fact above and equality (\ref{mult-norm}), taking modulus on both sides  yields
\begin{eqnarray}\label{eq:identity1}
\big|f(x)\big|^2
=\frac{1}{4}|A|^2+\frac{1}{4}\big|J\wedge I\big|^2|B|^2,
\end{eqnarray}
A simple calculation shows that
\begin{equation}\label{eq:identity2}
\begin{split}
|A|^2
=&\big(1+\langle I,J\rangle\big)^2\big|f(y)\big|^2
+\big(1-\langle I,J\rangle\big)^2\big|f(\bar y)\big|^2\\
&+2\big(1-\langle I,J\rangle^2\big)\big\langle f(y),f(\bar y)\big\rangle
\end{split}
\end{equation}
and
\begin{eqnarray}\label{eq:identity3}
|B|^2=\big|f(y)\big|^2+\big|f(\bar y)\big|^2
-2\big\langle f(y),f(\bar y)\big\rangle.
\end{eqnarray}
Notice that
\begin{eqnarray}\label{eq:identity4}
|J\wedge I|^2=1-\langle I,J\rangle^2.
\end{eqnarray}
Now inserting (\ref{eq:identity2}), (\ref{eq:identity3}) and (\ref{eq:identity4}) into (\ref{eq:identity1}) yields
$$\big|f(x)\big|^2 =\frac{1+\langle I,J\rangle}{2}\big|f(y)\big|^2+
\frac{1-\langle I,J\rangle}{2}\big|f(\bar y)\big|^2,$$
which completes the proof.
\end{proof}

\begin{remark}\label{remark}
The counterpart of the convex combination identity $(\ref{eq: modulus12})$ in Lemma \ref{eq:norm} also holds for slice regular functions defined on octonions  or more general real alternative algebras under the extra assumption that $f$ preserves at least one slice.  This can be verified similarly as in the proof of Proposition \ref{eq:C-version}, see \cite{Wang, RWX} for details.
\end{remark}

As a direct  consequence of Lemma \ref{eq:norm}, we conclude that  the maximum and  minimum moduli of $f$ are actually attained on the preserved slice.
\begin{corollary}\label{eq:modulus1}
Let $f$ be a slice monogenic  function on a symmetric slice domain  $U \subseteq \mathbb R^{n+1} $ such that $f(U_I)\subseteq \mathbb C_I $ for some $I\in \mathbb S $.  Then for each sphere $ u+v\mathbb S \subset U$, we have the following equalities:
$$\max_{J\in \mathbb S}\big|f(u+vJ)\big|=\max\Big(\big|f(u+vI)\big|,\big|f(u-vI)\big|\Big),$$
and
$$\min_{J\in \mathbb S}\big|f(u+vJ)\big|=\min\Big(\big|f(u+vI)\big|,\big|f(u-vI)\big|\Big).$$
\end{corollary}

We are now in a position to state the growth and distortion theorems for slice monogenic functions.

\begin{theorem}[Growth and Distortion Theorems for Paravectors]\label{th:DG-theorem1}
Let $f$ be a slice monogenic function on $\mathbb B$ such that its restriction $f_I$ to $\mathbb B_I$ is injective and $f(\mathbb B_I)\subseteq \mathbb C_I $ for some $I\in \mathbb S $. If $f(0)=0$ and $f'(0)=1$, then for all  $ x\in \mathbb B$, the following inequalities hold:
\begin{eqnarray}\label{eq:131}
\frac{|x|}{(1+|x|)^2}\leq |f(x)|\leq \frac{|x|}{(1-|x|)^2};
\end{eqnarray}
\begin{eqnarray}\label{eq:132}
\frac{1-|x|}{(1+|x|)^3}\leq |f'(x)|\leq \frac{1+|x|}{(1-|x|)^3} ;
\end{eqnarray}
\begin{eqnarray}\label{eq:133}
\frac{1-|x|}{1+|x|}\leq \big|xf'(x)\ast f^{-\ast}(x)\big|\leq \frac{1+|x|}{1-|x|}.
\end{eqnarray}
Moreover, equality holds for one of these six inequalities at some point $x_0\in \mathbb B\setminus\{0\}$ if and only if $f$ is of the form $$f(x)=x(1-xe^{I\theta})^{-\ast2},\quad\forall\,\, x \in \mathbb B,$$
for some $ \theta \in \mathbb R$.
\end{theorem}

\begin{proof}
Notice that
$f_I:\mathbb B_I\rightarrow \mathbb C_I $ is a univalent   function
by our assumption.
Theorem \ref{Th:dg-theorem} with $F$ replaced by $f_I$    implies  that the following inequalities
\begin{eqnarray}\label{eq: Growth}
\frac{|z|}{(1+|z|)^2}\leq |f(z)|\leq \frac{|z|}{(1-|z|)^2},
\end{eqnarray}
\begin{eqnarray}\label{eq: Distortion}
\frac{1-|z|}{(1+|z|)^3}\leq |f'(z)|\leq \frac{1+|z|}{(1-|z|)^3},
\end{eqnarray}
and
\begin{eqnarray}\label{eq: Distortion-Growth}
\frac{1-|z|}{1+|z|}\leq \bigg|\frac{zf'(z)}{f(z)}\bigg|\leq \frac{1+|z|}{1-|z|}
\end{eqnarray}
hold for every $z=u+vI\in \mathbb B_I.$ On the other hand, it follows from Lemma \ref{eq:norm} that
$$\big|f(x)\big|^2=\frac{1+\langle I,J\rangle}{2}\big|f(z)\big|^2+
\frac{1-\langle I,J\rangle}{2}\big|f(\bar z)\big|^2 $$
holds for every  $ x=u+vJ \in \mathbb B$. Since (\ref{eq: Growth}) holds for all $z=u+vI$, $\bar{z}=u-vI\in \mathbb B_I$, it immediately follows that inequalities in (\ref{eq:131}) hold for all  $ x=u+vJ \in \mathbb B$, in virtue of the convex combination identity above.   Inequalities in (\ref{eq:132}) can be proved in the same manner, since the condition that $f'(\mathbb B_I)\subseteq \mathbb C_I$ holds trivially so that Lemma \ref{eq:norm} can be used.

Now it remains to prove inequalities in $(\ref{eq:133})$. To this end, we first need to show that  the slice monogenic function  $xf'(x)\ast f^{-\ast}(x)$ is well-defined on the whole ball $\mathbb B$. We proceed as follows. First of all, since $f(0)=0$, by considering Taylor expansion of $f$ at the origin $0$ (see Theorem \ref{eq:Taylor}) and using the Cauchy-Hadamard formula for radius of convergence of power series (which is valid in the situation here by following the classical proof and making use of (\ref{mult-norm})), or by Remark \ref{R-extension}, we can write
\begin{equation}\label{f-g relation}
f(x)=xg(x),
\end{equation}
where $g$ is a slice monogenic function on $\mathbb B$.  This together with the injectivity of $f_I$  and $f'(0)=1$ implies that $g$ has no zeros on $\mathbb B_I$. Moreover,  $g$ maps $\mathbb B_I$ into $\mathbb C_I$, since $f$ does by our assumption.
Secondly, again from the assumption that $f(\mathbb B_I)\subseteq \mathbb C_I,$ i.e. all the coefficients of the Taylor expansion of $f$ at the origin belong to the complex plane $\mathbb C_I$, it follows that $$f_I^c(z)=\overline{f_I(\bar{z})},$$
and hence
\begin{equation}\label{def-fg}
N(f)_I(z)=f_I(z)\overline{f_I(\bar{z})}=z^2g_I(z)\overline{g_I(\bar{z})}=z^2N(g)_I(z).
\end{equation}
This implies that $N(f)(\mathbb B_I)\subseteq \mathbb C_I$.
 Furthermore, since $g$ maps $\mathbb B_I$ into $\mathbb C_I$ and has no zeros on $\mathbb B_I$, we obtain that $g^s_I$ is exactly
$g_I\overline{g_I(\, \bar{\cdot}\,)}$ and is zero free on $\mathbb B_I$. Thus it follows from Remark \ref{remarks on SM-Inverse} (ii) and \cite[Remark 2.6.8 and Lemma 2.5.12]{Co2} that $g^s$ is zero free on $\mathbb B$ as well. This together with the fact obtained easily from (\ref{f-g relation}) that
\begin{equation}\label{f-g symmetrization}
f^s(x)=x^2g^s(x), \qquad \forall\, x\in\mathbb B,
\end{equation}
implies that $0$ is the only zero of $f^s$. Therefore, according to Definition \ref{de: R-Inverse}, $f^{-\ast}$ and  $g^{-\ast}$ can be defined on $\mathbb B\setminus\{0\}$ and $\mathbb B$, respectively. Finally, in view of (\ref{f-g relation}),
\begin{equation}\label{f-g conjuagtion}
f^c(x)=xg^c(x), \qquad \forall\, x\in\mathbb B,
\end{equation}
from which and (\ref{f-g symmetrization}) it follows that
  the   relation
$$xf'(x)\ast f^{-\ast}(x)=(f'\ast g^{-\ast})(x)$$ holds for all $x\in \mathbb B\setminus\{0\}$. Since  the right-hand side is well-defined on the whole ball $\mathbb B$, the left-hand side can extend regularly to the whole ball $\mathbb B$, as desired.

Notice also that $xf'(x)\ast f^{-\ast}(x)$ is just the slice monogenic extension to $\mathbb B$ of the holomorphic function $z f'_I(z)/f_I(z)$, which also maps the unit disk $\mathbb B_I$ into $\mathbb C_I$. Now inequalities in $(\ref{eq:133})$ immediately follow from $(\ref{eq: Distortion-Growth})$ and
\begin{eqnarray*}
\Big|xf'\ast f^{-\ast}(x)\Big|^2 =\frac{1+\langle I,J\rangle}{2}\bigg|\frac{zf'(z)}{f(z)}\bigg|^2+\frac{1-\langle I,J\rangle}{2}\bigg|\frac{\bar{z}f'(\bar{z})}{f(\bar{z})}\bigg|^2,
\end{eqnarray*}
in view of  Lemma \ref{eq:norm}.

Furthermore, if equality holds for one of six inequalities in (\ref{eq:131}), (\ref{eq:132}) and (\ref{eq:133}) at some point $x_0=u_0+v_0J\neq 0$ with $J \in \mathbb S$, then the corresponding equality also holds at $z_0=u_0+v_0I$ or $\bar{z}_0=u_0-v_0I$. Then from Theorem \ref{Th:dg-theorem}, we obtain that
$$f_I(z)=\frac{z}{(1-e^{I\theta} z)^2}\ ,\qquad \forall\,\, z\in \mathbb B_I,$$ for some $ \theta \in \mathbb R$, which implies
$$f(x)=x(1-xe^{I\theta})^{-\ast2},\qquad\forall \,\,x \in \mathbb B.$$
The converse part is obvious.
Now the proof is complete.
\end{proof}

\begin{remark}
The right-hand inequalities in $(\ref{eq:131})$ and $(\ref{eq:132})$ can follow alternatively from the well-known and highly non-trivial Bieberbach-de Branges theorem for univalent functions on the open unit disk $\mathbb D\subset\mathbb C$.
\end{remark}

Let  $F: \mathbb D\rightarrow \mathbb C$ be a  univalent  function on the unit disc $\mathbb D$ of the complex  plane with Taylor expansion
$$F(z)=z+\sum_{m=2}^\infty z^m a_m, \qquad a_m\in \mathbb C.$$
We consider the  canonical  imbedding $\mathbb C\subset \mathbb  R^{n+1}$ by expanding the basis $\{1, i\}$ of $\mathbb C$ to the basis $\{1, e_1, \ldots, e_n \}$ of $\mathbb R^{n+1}$ with $e_1=i$. Therefore we can construct a natural extension of $F$ to $\mathbb B$ by setting
$$f(x)=x+\sum_{m=2}^\infty x^m a_m, \qquad x\in \mathbb B.$$
It is evident that
 $f$ is a slice monogenic function on the open unit ball $\mathbb B=B(0,1)$ such that its restriction $\left.f\right|_{\mathbb D}=F$  is injective and satisfies that $F(\mathbb D)\subseteq \mathbb C $. Clearly,  $f(0)=0$ and $f'(0)=1$.
 Thus $f$ satisfies all  the assumptions of Theorem
\ref{th:DG-theorem1} and thus Theorem \ref{main-thm-Clifford} immediately follows.

\begin{remark}

The slice monogenic extension of  holomorphic functions on the unit disc $\mathbb D$ of the complex plane can result in the theory of slice monogenic elementary functions. We refer  to  \cite{Co2} for the corresponding functional calculus and applications.

\end{remark}

The following proposition is of independent interest.
\begin{proposition}\label{injectivity}
Let $f$ be a slice monogenic function on a symmetric slice domain  $U \subseteq \mathbb R^{n+1} $ such that its restriction $f_I$ to $U_I$ is injective and $f(U_I)\subseteq \mathbb C_I $ for some $I\in \mathbb S $. Then the restriction $f_J:U_J\rightarrow \mathbb R_n$ is also injective for every $J\in \mathbb S$.
\end{proposition}

\begin{proof}
Suppose that there are two points $x=\alpha+\beta J$ and $y=\gamma+\delta J$ such that $f(x)=f(y)$, it suffices to prove that $x=y$. If $J=\pm I$, the result follows from the assumption. Otherwise, from Theorem \ref{eq:formula} one can deduce that
\begin{equation*}
f(x)=\frac{1}{2}\big(f(z)+f(\bar z)\big)-\frac{1}{2}JI\big(f(z)-f(\bar z)\big)
\end{equation*}
and
\begin{equation*}
f(y)=\frac{1}{2}\big(f(w)+f(\bar w)\big)-\frac{1}{2}JI\big(f(w)-f(\bar w)\big).
\end{equation*}
Here $z=\alpha+\beta I$ and $w=\gamma+\delta I$ for the given $I\in \mathbb S$.
Therefore,
$$\Big(\big(f(z)+f(\bar z)\big)-\big(f(w)+f(\bar w)\big)\Big)-
JI\Big(\big(f(z)-f(\bar z)\big)-\big(f(w)-f(\bar w)\big)\Big)=0.$$
Since $f(U_I)\subseteq \mathbb C_I$, 1 and $J$ are linearly independent on $\mathbb C_I$ we obtain that
$$f(z)+f(\bar z)=f(w)+f(\bar w)$$
and
$$f(z)-f(\bar z)=f(w)-f(\bar w),$$
which implies that
$f(z)=f(w)$. Thus it follows from the injectivity of $f_I$ that $z=w$ and consequently, $x=y$.
\end{proof}
\begin{remark}
Let $f$ be as described in Theorem \ref{th:DG-theorem1}. Then $f_J:\mathbb B_J\rightarrow \mathbb R_n$ is injective for any $J\in \mathbb S$ by the preceding proposition. Unfortunately, the authors do not know whether $f:U\rightarrow \mathbb R_n$ is injective.
\end{remark}
\section{Growth, Distortion and Covering Theorems for  Slice Regular Functions} \label{GDTs for slice regularity}
Let $\mathbb H$ denote the non-commutative, associative, real algebra of quaternions with standard basis $\{1,\,i,\,j, \,k\}$,  subject to the multiplication rules
$$i^2=j^2=k^2=ijk=-1.$$
Let $\langle$ , $\rangle$ denote  the standard inner product on $\mathbb H\cong\mathbb R^4$, i.e.
$$\langle p,q\rangle={\rm{Re}}(p\bar{q})=\sum\limits_{n=0}^3x_ny_n$$
for any $p=x_0+x_1i+x_2j+x_3k$, $q=y_0+y_1i+y_2j+y_3k\in \mathbb H$.

In this section, we shall consider  slice regular functions defined on domains in  quaternions $\mathbb H$ with values also in $\mathbb H$. These functions are not slice  monogenic functions  obtained by setting $n=2$ in the Clifford algebra $\mathbb R_n$. Such a class of functions enjoys many nice properties similar to those of classical holomorphic functions of one complex variable.
For example, the open mapping theorem holds for slice regular functions on symmetric slice domains in $\mathbb H$, but fails for  slice  monogenic functions even in the quaternionic setting. A simple counterexample is the imbedding  map $\imath: \mathbb R^3\hookrightarrow \mathbb R_2\simeq\mathbb H$. The open mapping theorem allows us to prove the Koebe type one-quarter theorem (see Theorem \ref{th:Koebe-theorem} below).  Furthermore,  only in the quaternionic setting, we have  an explicit formula to express of the regular product and regular quotient in terms of  the usual pointwise product and quotient. It is exactly this explicit formula which plays a crucial role in many arguments, see the monograph \cite{GSS} and the recent papers \cite{RW2, Wang} for more details. In higher dimensions, the formulas to express slice product and slice quotient in  terms of the usual pointwise ones hold true only under some special cases, see \cite[Corollary 3.5 and Theorem 3.7]{Ghiloni2} for details.
This phenomenon distinguishes in a certain sense quaternions from other real alternative algebras.

To introduce the theory of slice regular functions, we will denote by $\mathbb S$ the unit $2$-sphere of purely imaginary quaternions, i.e.,
$$\mathbb S=\Big\{q\in\mathbb H:q^2=-1\Big\}.$$
For every $I \in \mathbb S $ we will denote by $\mathbb C_I$ the plane $ \mathbb R \oplus I\mathbb R $, isomorphic to $ \mathbb C$, and, if $\Omega \subseteq \mathbb H$, by $\Omega_I$ the intersection $ \Omega \cap \mathbb C_I $. Also, we will denote by $\mathbb B$ the open unit ball centred at the origin in $\mathbb H$, i.e.,
$$\mathbb B=\Big\{q \in \mathbb H:|q|<1\Big\}.$$

We can now recall the definition of slice regularity.
\begin{definition} \label{de: regular} Let $\Omega$ be a domain in $\mathbb H$. A function $f :\Omega \rightarrow \mathbb H$ is called \emph{slice} \emph{regular} if, for all $ I \in \mathbb S$, its restriction $f_I$ to $\Omega_I$ is \emph{holomorphic}, i.e., it has continuous partial derivatives and satisfies
$$\bar{\partial}_I f(x+yI):=\frac{1}{2}\left(\frac{\partial}{\partial x}+I\frac{\partial}{\partial y}\right)f_I (x+yI)=0$$
for all $x+yI\in \Omega_I $.
 \end{definition}

The notions of slice domain, of symmetric slice domain and of slice derivative are similar to those already given in Section 2. Moreover, the corresponding results still hold for the slice regular functions in the setting of quaternions, such as the splitting lemma, the representation formula, the power series expansion and so on.

Now we can establish  the following result by some obvious modifications of the proof of Proposition \ref{eq:C-version}.

\begin{proposition}\label{eq:modulus150}
Let $f$ be a slice regular function on a symmetric slice domain  $\Omega \subseteq \mathbb H$. Then for every $q=x+yJ \in \Omega $ and every $I\in \mathbb S$,  there   holds the identity
\begin{equation}\label{eq:identity of quaternionic version}
\big|f(q)\big|^2=\frac{1+\langle I,J\rangle}{2}\big|f(z)\big|^2+
\frac{1-\langle I,J\rangle}{2}\big|f(\bar z)\big|^2-
\Big\langle {\rm{Im}} \big(f(z)\overline{f(\bar z)}\big),\, I\wedge J\Big\rangle,
\end{equation}
where  $z=x+yI$ and $\bar z=x-yI$.
\end{proposition}

Before presenting the key ingredient of   establishing the growth and distortion theorems, we first make an equivalent characterization of the vanishing of the third term on the right-hand side of (\ref{eq:identity of quaternionic version}), thanks to the speciality of quaternions.
\begin{theorem}\label{eq:SN-Condition}
Let $f$ be a slice regular function on a symmetric slice domain  $\Omega \subseteq \mathbb H$ and let $I\in\mathbb S$. Then
$$\Big\langle {\rm{Im}} \big(f(z)\overline{f(\bar z)}\big), \,I\wedge J\Big\rangle=0$$
for all $J\in \mathbb S$ and all $z\in \Omega_I$ if and only if there exist $u\in \partial\mathbb B$ and a slice regular function $g$ on $\Omega$ that preserves the slice $\Omega_I$ such that
$$f(q)=g(q)u$$
on $\Omega$.
\end{theorem}
\begin{proof}
 We only prove the necessity, since the sufficiency is obvious.
Let $f_I(z)=F(z)+G(z)K$ be the splitting of $f$ with $I, K\in\mathbb S$ and  $I\perp K$, and $F, G:\Omega_I\rightarrow\mathbb C_I$ two holomorphic functions. Take $L\in \mathbb S$ such that $\{1,I,K,L\}$ is an orthonormal basis of quaternions $\mathbb H$ and let $V$ denote the real vector space generated by the set  $\big\{I\wedge J:J\in \mathbb S\big\}$. Then it is clear that
\begin{eqnarray}\label{de:V}
V=K\mathbb R\oplus L\mathbb R.
\end{eqnarray}
Moreover, a simple calculation gives
$$f(z)\overline{f(\bar z)}=\Big(F(z)\overline{F(\bar z)}+G(z)\overline{G(\bar z)}\Big)+\Big(F(\bar z)G(z)-F(z)G(\bar z)\Big)K,$$
from which and (\ref{de:V}) it follows that
$$\Big\langle {\rm{Im}} \big(f(z)\overline{f(\bar z)}\big), I\wedge J\Big\rangle=0,\qquad \forall \,J\in \mathbb S,$$ if and only if
\begin{eqnarray}\label{eq:FG}
F(z)G(\bar{z})=F(\bar{z})G(z),\qquad \forall \,z\in \Omega_I.
\end{eqnarray}
If $G\equiv 0$ on $\Omega_I$, there is nothing  to prove and  the desired result follows. Otherwise, $G\not\equiv 0$, let $\mathcal{Z}_G$ denote the zero set of $G$, then by the identity principle $\mathcal{Z}_G$ has no accumulation points in $\Omega_I$ and so does $\mathcal{\overline{Z}}_G:=\big\{\bar z\in \Omega_I: z\in \mathcal{Z}_G \big\}$. Therefore by (\ref{eq:FG}),
$$\frac{F(z)}{G(z)}=\frac{F(\bar{z})}{G(\bar{z})}$$
is both holomorphic and anti-holomorphic on $\Omega_I\setminus \big(\mathcal{Z}_G\cup\mathcal{\overline{Z}}_G\big)$, which is still a domain of $\mathbb C_I$, thus there exists a constant $\lambda\in \mathbb C_I$ such that
$$\frac{F(z)}{G(z)}=\frac{F(\bar{z})}{G(\bar{z})}=\lambda,$$
which implies that
$F=\lambda G$ on $\Omega_I\setminus \big(\mathcal{Z}_G\cup\mathcal{\overline{Z}}_G\big)$ and hence on $\Omega_I$ by the identity principle.

Now let
$$g=:(1+|\lambda|^{2})^{\frac12}{\rm{ext}}(G),$$
and set $$u=:(1+|\lambda|^{2})^{-\frac12}(\lambda+K)\in \partial \mathbb B.$$
Then $g$ is a slice regular function on $\Omega$ such that $g(\Omega_I)\subseteq \mathbb C_I$  and $f=gu$, which completes the proof.
\end{proof}

%The precious theorem shows that when $f$ preserves one slice up to a suitable rotation from the group
%$$\mathbb S^3:=\bigg\{\left(
%\begin{array}{cc}
% u & 0 \\
% 0 & 1 \\
%   \end{array}
%   \right)
%:u\in \partial\mathbb B\bigg\}\leq Sp(1,1),$$ i.e. $f(\Omega_I)\subseteq \mathbb C_Iu $ for some $I\in\mathbb S$ and some $u\in \partial\mathbb B$, the squared norm of $f$ can thus be expressed as a convex  combination of those in the preserved slice.

As a direct consequence, we obtain

\begin{corollary}\label{eq:modulus relation151}
Let $I$ be an element of  $\mathbb S $ and $f$  a slice regular function on a symmetric slice domain  $\Omega \subseteq \mathbb H $.   Then the following convex combination identity
\begin{equation}\label{eq: modulus2}
\big|f(x+yJ)\big|^2 =\frac{1+\langle I,J\rangle}{2}\big|f(x+yI)\big|^2+
\frac{1-\langle I,J\rangle}{2}\big|f(x-yI)\big|^2
\end{equation}
holds for every $ x+yJ \in \Omega $ if and only if there exists some $u\in \partial\mathbb B$
such that $f(\Omega_I)\subseteq \mathbb C_Iu $.
\end{corollary}

In particular, each element $f$ from  the slice regular automorphism group of the open unit ball $\mathbb B$ of $\mathbb H$
$$\textrm{Aut}(\mathbb B)=\Big\{f(q)=(1-q\bar a)^{-\ast}\ast(q-a)u:a\in\mathbb B,\,u\in \partial\mathbb B\Big\}$$
satisfies  the condition that there exists some $u\in \partial\mathbb B$
such that $f(\Omega_I)\subseteq \mathbb C_Iu $ so that  equality (\ref{eq: modulus2}) holds for such an $f$.

From Corollary \ref{eq:modulus relation151}, we also conclude that  the maximum and  minimum moduli of every slice regular functions on a symmetric slice  domains in $\mathbb H$ that preserves one slice  are actually attained on its preserved slice.
\begin{corollary}\label{eq:modulus relation152}
Let $f$ be a slice regular function on a symmetric slice domain  $\Omega \subseteq \mathbb H $ such that $f(\Omega_I)\subseteq \mathbb C_I $ for some $I\in \mathbb S $.  Then for each sphere $ x+y\mathbb S \subset \Omega$, the following equalities hold:
\begin{equation}\label{max-ineq}
\max_{J\in \mathbb S}\big|f(x+yJ)\big|=\max\Big(\big|f(x+yI)\big|,\big|f(x-yI)\big|\Big),
\end{equation}
and
\begin{equation}\label{min-ineq}
\min_{J\in \mathbb S}\big|f(x+yJ)\big|=\min\Big(\big|f(x+yI)\big|,\big|f(x-yI)\big|\Big).
\end{equation}
Consequently,
\begin{equation}\label{Sup-ineq}
\sup_{q\in \Omega}|f(q)|=\sup_{z\in \Omega_I}|f(z)|,
\end{equation}
\begin{equation}\label{Inf-ineq}
\inf_{q\in \Omega}|f(q)|=\inf_{z\in \Omega_I}|f(z)|.
\end{equation}
\end{corollary}

\begin{remark}
Equalities $(\ref{max-ineq})$ and $(\ref{min-ineq})$ were first proved in \cite[Proposition 1.13]{Sarfatti} and \cite[Proposition 2.6]{dGS}.  Together with the classical growth and distortion theorems,  Corollary \ref{eq:modulus relation152} is sufficient to prove Theorem \ref{th:DG-theorem153} below even without Corollary  \ref{eq:modulus relation151}. Despite this trivial fact, Corollary \ref{eq:modulus relation151} is of independent interest and has its intrinsic value.
 It presents  additionally  a new convex combination identity (\ref{eq: modulus2}) and provides a sufficient and necessary condition under which (\ref{eq: modulus2}) holds identically.  This convex combination identity  is also quite useful for other purposes. For instance, it provides an effective approach to a quaternionic version of a well-known Forelli-Rudin estimate, which will play a fundamental role in the theory of various spaces of slice regular functions \cite{RX}.
\end{remark}

Now we come to state the growth and distortion theorems for slice regular functions.

\begin{theorem}[Growth and Distortion Theorems for Quaternions]\label{th:DG-theorem153}
Let $f$ be a slice regular function on $\mathbb B$ such that its restriction $f_I$ to $\mathbb B_I$ is injective and $f(\mathbb B_I)\subseteq \mathbb C_I $ for some $I\in \mathbb S $. If $f(0)=0$ and $f'(0)=1$, then for all  $ q\in \mathbb B$, the following inequalities hold:
\begin{eqnarray}\label{eq:111}
\qquad\ \frac{|q|}{(1+|q|)^2}\leq |f(q)|\leq \frac{|q|}{(1-|q|)^2};
\end{eqnarray}
\begin{eqnarray}\label{eq:112}
\qquad\ \frac{1-|q|}{(1+|q|)^3}\leq |f'(q)|\leq \frac{1+|q|}{(1-|q|)^3} ;
\end{eqnarray}
\begin{eqnarray}\label{eq:113}
\qquad\ \frac{1-|q|}{1+|q|}\leq \big|qf'(q)\ast f^{-\ast}(q)\big|\leq \frac{1+|q|}{1-|q|}.
\end{eqnarray}
Moreover, equality holds for one of these six inequalities at some point $q_0\in \mathbb B\setminus \{0\}$ if and only if $f$ is of the form $$f(q)=q(1-qe^{I\theta})^{-\ast2},\quad\forall\,\, q \in \mathbb B, $$
for some $ \theta \in \mathbb R$.
\end{theorem}

Let  $F: \mathbb D\rightarrow \mathbb C$ be a univalent   function on the unit disc $\mathbb D$ of the complex  plane with Taylor expansion
$$F(z)=z+\sum_{n=2}^\infty z^n a_n, \qquad a_n\in \mathbb C.$$
As in Section 3, with a  canonical  imbedding $\mathbb C\subset \mathbb H$,  we can construct a natural slice regular extension of $F$ to $\mathbb B$ via
$$f(q)=q+\sum_{n=2}^\infty q^n a_n, \qquad q\in \mathbb B.$$
It is evident that
 $f$ is a slice regular function on the open unit ball $\mathbb B=B(0,1)$ such that its restriction $\left.f\right|_{\mathbb D}=F$  is injective and satisfies that $F(\mathbb D)\subseteq \mathbb C $. Clearly,  $f(0)=0$ and $f'(0)=1$.
 Thus $f$ satisfies all the  assumptions of Theorem
\ref{th:DG-theorem153} and this results in the following theorem.

\begin{theorem}\label{main-thm-quarternion} Let  $F: \mathbb D\rightarrow \mathbb C$ be a  univalent  function on  $\mathbb D$ such that $F(0)=0$ and $F'(0)=1$, and let
$f:\mathbb B\rightarrow \mathbb H$ be the slice regular extension of $F$. Then for all $q\in\mathbb B,$ the following inequalities hold:
\begin{eqnarray*}\label{eq:12}
\frac{|q|}{(1+|q|)^2}\leq |f(q)|\leq \frac{|q|}{(1-|q|)^2};
\end{eqnarray*}
\begin{eqnarray*}
\frac{1-|q|}{(1+|q|)^3}\leq |f'(q)|\leq \frac{1+|q|}{(1-|q|)^3} ;
\end{eqnarray*}\label{eq:11}
\begin{eqnarray*}\label{eq:12}
\frac{1-|q|}{1+|q|}\leq \big|qf'(q)\ast f^{-\ast}(q)\big|\leq \frac{1+|q|}{1-|q|}.
\end{eqnarray*}
Moreover, equality holds for one of these six inequalities at some point $q_0\in \mathbb B\setminus \{0\}$ if and only if
$$f(q)=q(1-qe^{i\theta})^{-\ast2},\qquad\forall \,\,q \in \mathbb B.$$
\end{theorem}

Next we  digress to Koebe one-quarter theorem for slice regular functions  on the open unit ball $\mathbb B\subset \mathbb H$. We recall the following definition (see \cite[Definition 7.5]{GSS}).
\begin{definition}
Let $f$ be  a slice regular function on a symmetric slice domain $\Omega\subset \mathbb H$.  The \textit{degenerate set} of $f$ is defined to be the union $D_f$ of the $2$-dimensional spheres $S=x+y\mathbb S$ (with $y\neq 0$) such that $\left.f\right|_S$ is constant.
\end{definition}

Now as a direct consequence of the open mapping theorem and the first inequality in $(\ref{eq:111})$, we have the following result, which is a generalization of
\cite[Theorem 3.11 (1)]{GGCO}.

\begin{theorem}[Koebe One-Quarter Theorem]\label{th:Koebe-theorem}
Let $f$ be a slice regular function on   $\mathbb B$ such that its restriction $f_I$ to $\mathbb B_I$ is injective and $f(\mathbb B_I)\subseteq \mathbb C_I $ for some $I\in \mathbb S $. If $f(0)=0$ and $f'(0)=1$, then it holds that
$$B(0,\frac14)\subset f(\mathbb B).$$
\end{theorem}

\begin{proof}
By assumption, the degenerate set $D_f$ of $f$ is empty. Then $f$ is open by the open mapping theorem (see \cite[Theorem 7.7]{GSS}). This together with the first inequality in $(\ref{eq:111})$ shows that the image set $f(\mathbb B)$, containing the origin $0$, is an open subset  of $\mathbb H$, whose boundary $\partial f(\mathbb B)$ lies outside of the ball $B(0,1/4)$. Indeed, for each point $w\in \partial f(\mathbb B)$, there exists a sequence $\{q_n\}_{n=1}^{\infty}$ in $\mathbb B$ such that $\lim_{n\rightarrow \infty}f(q_n)=w$. By passing to a subsequence, we may assume that  the sequence $\{q_n\}_{n=1}^{\infty}$ itself converges to one point, say $q_{\infty}\in \overline{\mathbb B}$. By the openness of $f$, $q_{\infty}$ must lie on the boundary $\partial \mathbb B$. Thus in view of the first inequality in $(\ref{eq:111})$,
$$|w|=\lim_{n\rightarrow \infty}|f(q_n)|\geq\lim_{n\rightarrow \infty}\frac{|q_n|}{(1+|q_n|)^2}=\frac14.$$
Consequently, $f(\mathbb B)$ must contain the ball $B(0,1/4)$. This completes the proof.
\end{proof}

Let $\mathcal{SR}(\mathbb B)$ denote the set of slice regular functions on the open unit ball $\mathbb B\subset\mathbb H$.
We define
$$\mathcal{S}:=\Big\{f\in \mathcal{SR}(\mathbb B):\exists\ I\in\mathbb S \ \mbox{such that} \ f_I \ \mbox{is injective and}\  f_I(\mathbb B_I)\subseteq \mathbb C_I\Big\}$$
and
$$\mathcal{S}_0:=\Big\{f\in \mathcal{S}: f(0)=0,\, f'(0)=1\Big\}.$$
For each $f\in \mathcal{S}_0$, we use $r_0(f)$ to denote the radius of the smallest ball $B(0,r)$ contained in $f(\mathbb B)$. Also for every $\theta\in \mathbb R$ and every $I\in\mathbb S$, denote by $k_{I, \theta}$ the slice regular function given by
\begin{equation}\label{Extension of Koebe functions}
k_{I, \theta}(q)=q(1-qe^{I\theta})^{-\ast2},\quad\forall\,\, q \in \mathbb B,
\end{equation}
which obviously belongs to the class $\mathcal{S}_0$. The image set of the unit disc $\mathbb B_I$ under $k_{I, \theta}$ is exactly  the complex plane except for a radial slit from $\infty$ to $-e^{I\theta}/4$. This fact together with Theorem $\ref{th:Koebe-theorem}$ gives the following result.

\begin{theorem}
With notations as above, the following statements hold:
\begin{enumerate}[leftmargin=1.7pc, parsep=4pt,label=(\roman*)]
 \item [{\rm{(i)}}] for each $f\in \mathcal{S}_0$,
  $$r_0(f)\geq \frac14$$
  with equality if and only if $f=k_{I, \theta}$ for some $I\in\mathbb S$ and some $\theta \in \mathbb R$;
  \item [{\rm{(ii)}}] $$\bigcap_{f\in \mathcal{S}_0}f(\mathbb B)=B(0,\frac14).$$
\end{enumerate}
\end{theorem}

\begin{proof}
We only prove ${\rm{(i)}}$. It suffices to consider the extremal case, since the remaining is clear. If $r_0(f)=1/4$, from the proof of Theorem \ref{th:Koebe-theorem} and Inequality $(\ref{Inf-ineq})$,  we conclude that there exists some $I_0\in\mathbb S$ such that $1/4$ is exactly the radius of the smallest disc $\mathbb B_{I_0}(0,r)$ contained in the image set $f_{I_0}(\mathbb B_{I_0})$  of the unit disk $\mathbb B_{I_0}$ under the classical univalent   function $f_{I_0}:\mathbb B_{I_0}\rightarrow \mathbb C_{I_0}$. This is possible only if $f=k_{I_0, \theta}$ for some $\theta\in \mathbb R$ (see the proof of \cite[Theorem 1.1.5]{GG} or \cite[Theorem 2.3]{Duren}). Now the proof is complete.
\end{proof}

\begin{remark}
Two remarks are in order:
\begin{enumerate}[leftmargin=1.7pc, parsep=4pt,label=(\roman*)]
 \item [{\rm{(i)}}]
It is noteworthy here that Gal et al. \cite{GGCO} dealt with the growth, distortion and covering theorems for \textit{slice preserving} and \textit{injective} slice regular functions on the open unit ball $\mathbb B\subset \mathbb H$ with certain normalized conditions. More precisely, they focused on \textit{injective} slice functions $f$ on $\mathbb B$ of the form
 $$f(q)=q+\sum_{n=2}^\infty q^n a_n$$
 with $\{a_n\}_{n\geq2}$ being a sequence of \textit{real numbers}; see \cite[Theorem 3.11]{GGCO} for details.  While, in the present paper we consider   slice regular functions $f(q)=q+\sum_{n=2}^\infty q^n a_n$ on $\mathbb B$ for which there exists \textit{some} $I\in\mathbb S$ such that the restriction \textit{$f_I$ is injective} and  $\{a_n\}_{n\geq2}$ is  a sequence of \textit{numbers in the complex plane $\mathbb C_I$ determined  by $I$.}
 Thus our result  properly includes the former case. Moreover, our approach to the Koebe type one-quarter theorem (Theorem  \ref{th:Koebe-theorem}), which can be specialized to the complex case,  depends only on the open mapping theorem and the first inequality in $(\ref{eq:111})$, and does not involve  compositions of functions.  We refer the reader to  \cite[p.14]{GG}  and to \cite[p.31]{Duren} for a standard proof of the classical Koebe type one-quarter theorem for univalent functions.

 \item [{\rm{(ii)}}] Functions $k_{I, \theta}$ of the form in $(\ref{Extension of Koebe functions})$ are specific examples in $\mathcal{S}_0$.  In view of Theorem \ref{th:Koebe-theorem},   the image of $\mathbb B$ under the function $k_{I, \pi/2}$ contains the open ball $B(0,1/4)$. However,  it seems not so easy to directly deduce this fact from the classical complex result, without using  the open mapping theorem and the first inequality in $(\ref{eq:111})$.
\end{enumerate}
\end{remark}

The following proposition  is the quaternionic version of Proposition \ref{injectivity} for slice regular functions.

\begin{proposition}\label{injectivity01}
Let $f$ be a slice regular function on a symmetric slice domain  $\Omega \subseteq \mathbb H$ such that its restriction $f_I$ to $\Omega_I$ is injective and $f(\Omega_I)\subseteq \mathbb C_I $ for some $I\in \mathbb S $. Then its restriction $f_J:\Omega_J\rightarrow \mathbb H$ is also injective for every $J\in \mathbb S$.
\end{proposition}

\begin{remark}

Let $f$ be as described in Theorem \ref{th:DG-theorem153}. Then according to the preceding proposition, $f_J:\mathbb B_J\rightarrow \mathbb H$ is injective for every $J\in \mathbb S$. It is well worth knowing whether $f:\mathbb B\rightarrow\mathbb H$ is injective. If it is indeed the case, together with the first inequality in $(\ref{eq:111})$ and invariance of domain theorem, it would provide an alternative approach to Theorem \ref{th:Koebe-theorem}.
\end{remark}

\section{Concluding remarks}\label{Concluding remarks}
As pointed out in Remark \ref{remark}, the counterpart of the convex combination identity $(\ref{eq: modulus12})$ in Lemma \ref{eq:norm} also holds for slice regular functions defined on octonions  or more general real alternative algebras under the extra assumption that $f$ preserves at least one slice. Therefore some of the results given in the preceding sections can be easily generalized by slight modification to these new settings. Finally, we conclude with an open question connected with the  subject of the present paper.

Recall that $\mathcal{SR}(\mathbb B)$ is the set of slice regular functions on the open unit ball $\mathbb B\subset\mathbb H$.
We denote
$$\mathcal{SR}_0(\mathbb B):=\Big\{f\in \mathcal{SR}(\mathbb B): f(0)=0,\, f'(0)=1\Big\}$$
and
$$\mathcal{S}_0:=\Big\{f\in \mathcal{SR}_0(\mathbb B):  \exists\ I\in\mathbb S \ \mbox{such that} \ f_I \ \mbox{is injective and}\  f_I(\mathbb B_I)\subseteq \mathbb C_I\Big\}.$$

$\mathbf{Open \;\;question:}${\footnote{Very recently, Xu has found a negative number to this question, see \cite[Example 3.1,  Theorems 5.1 and 5.6]{Xu2016} for details.}} Is the  class $\mathcal{S}_0$    the largest subclass of $\mathcal{SR}_0(\mathbb B)$ in which the corresponding growth, distortion  and covering  theorems hold?

\section*{Acknowledgments}
The authors  thank sincerely the anonymous referee for his/her careful reading of this paper and several  valuable suggestions and comments, which have significantly improved the presentation of this paper.

\bibliographystyle{amsplain}

\end{document}